\documentclass[11pt,a4paper,reqno]{amsart}
\usepackage[latin1]{inputenc}
\usepackage[T1]{fontenc}
\usepackage{lmodern}
 \usepackage[english]{babel,minitoc}
\usepackage{hyperref}
\usepackage{amsfonts}
\usepackage{amsmath}

\numberwithin{equation}{section}\theoremstyle{plain}
\newtheorem{theorem}{Theorem}[section]

\newtheorem{lemma}[theorem]{Lemma}
\newtheorem{corollary}[theorem]{Corollary}
\newtheorem{definition}[theorem]{Definition}
\newtheorem{remark}[theorem]{Remark}

\address{\begin{center}{\small Department of Mathematics and Computer Sciences, Faculty of Sciences,\\
Equipe d'Analyse Harmonique et Probabilit\'{e}s, University Moulay Isma\"{\i}l,\\
BP 11201 Zitoune, Mekn\`{e}s, Morocco}
\end{center}}

\begin{document}

\title[Uncertainty Principles For the continuous  Gabor quaternion linear canonical transform]
{Uncertainty Principles For The continuous  Gabor quaternion linear canonical transform  }

\author[ M. El kassimi and S. Fahlaoui ]{ Mohammed El kassimi and Sa\"{\i}d Fahlaoui  }

\address{Mohammed El kassimi} \email{m.elkassimi@edu.umi.ac.ma}

\address{Sa\"{\i}d Fahlaoui}  \email{s.fahlaoui@fs.umi.ac.ma}

\maketitle
\begin{abstract}
Gabor transform is one of the performed tools for time-frequency signal analysis. The principal aim of this paper is to generalize the Gabor Fourier transform to the quaternion linear canonical transform. Actually, this transform gives us more flexibility to studied nonstationary and local signals associated with the quaternion linear canonical transform. Some useful properties are derived, such as Plancherel and inversion formulas. And we prove some uncertainty principles: those including Heisenberg's, Lieb's and logarithmic inequalities. We finish by analogs of concentration and Benedick's type theorems.
\end{abstract}
 {\it \textbf{keywords}:} {Gabor transform, Quaternion algebra, linear canonical transform, Heisenberg inequality, Lieb inequality, Logarithmic inequality, Benedick theorem.}
\section*{Introduction}
The linear canonical transform (LCT) was first introduced in 1970 \cite{Moshi, coll}, many transforms are the special cases of LCT, such as the classical Fourier transform, the Lorentz transform and  the Fresnel transform. Due to this generalization, the LCT has recently received attention in signal processing, optics and radar analysis\cite{Ozakyas,Pei}.\\
In \cite{Morais,Kou1} the authors give a generalization of the LCT to the quaternionic case. The quaternionic linear canonical transform (QLCT) has large applications in signal processing and pattern recognition \cite{Alieva,Ying,John}.\\
However, the QLCT is insufficient to studied the local signal, and non stationary  signals. To resolve this problem, in 1940 Dennis Gabor \cite{Gabor1} introduce the Gabor transform (GT), it is also known as the windowed transform. The GT becomes a powerful tool in time and frequency analysis of signal processing, it allows  detection and estimation of localized time-frequency. In addition, it used for filtering and modifying the signal in the local region.\\
In this paper, we are going to generalize the GT to the QLCT, we obtain the Gabor quaternion linear canonical Fourier transform. The paper is organized as follows, in the  second  section \ref{section2}, we recall the main Harmonic analysis properties for the QLCT. In section \ref{section3} we give the definition of the GQLCT, and we demonstrate some important results for it. In section \ref{section4}, we generalize the Lieb's and Logarithmic's inequalities to the GQLCT. The last section \ref{section5} is devoted to give the Heisenberg uncertainty principle and to prove some concentration theorem and local uncertainty principles.

\section{Definition and properties of quaternion $\mathbb{H}$}\label{section2}

The quaternion algebra $\mathbb{H}$ is defined over $\mathbb{R}$ with three imaginary units i, j and k obey the Hamilton's multiplication rules, \begin{equation}\label{equation1} ij=-ji=k, ~~ jk=-kj=i, ~~ki=-ik=j \end{equation}
\begin{equation}
 i^2=j^2=k^2=ijk=-1
\end{equation}
According to \ref{equation1} $\mathbb{H}$ is non-commutative, one cannot directly extend various results on complex numbers to a quaternion. For simplicity, we express a quaternion q as the sum of scalar $q_1$, and a pure 3D quaternion q.
Every quaternion can be written explicitly as:
\begin{equation*}
  q=q_1+iq_2+jq_3+kq_4\in \mathbb{H},~~ q_1,q_2,q_3,q_4\in\mathbb{R},
\end{equation*}

The conjugate of quaternion $q$ is obtained by changing the sign of the pure part, i.e.
\begin{equation*}
 \overline{q}=q_1-iq_2-jq_3-kq_4
\end{equation*}
The quaternion conjugation is a linear anti-involution
\begin{equation*}
 \overline{\overline{p}}=p,~~~~ \overline{p+q}=\overline{p}+\overline{q},~~ \overline{pq}=\overline{q} \overline{p},~ \forall p,q\in \mathbb{H}
\end{equation*}
The modulus $|q|$ of a quaternion q is defined as
\begin{equation*}
 |q|=\sqrt{q\overline{q}}=\sqrt{q_1^2+q_2^2+q_3^2+q_4^2},~~ |pq|=|p||q|.
\end{equation*}
It is straight forward to see that
\begin{equation*}
|pq|=|p||q|, |q|=|\overline{q}|, p, q \in \mathbb{H}
\end{equation*}

In particular, when $q=q_1$ is a real number, the module $|q|$ reduces to the ordinary Euclidean modulus, i.e. $|q|=\sqrt{q_1q_1}$.  A function $f:\mathbb{R}^{2}\rightarrow\mathbb{H}$ can be expressed as
\begin{equation*}
f (x, y): =f_1 (x, y) +if_ {2} (x, y) +jf_3 (x, y)+kf_4 (x, y),
\end{equation*}

 where $(x,y)\in\mathbb{R}\times\mathbb{R}$.\\
We introduce an inner product of functions $f, g$ defined on $\mathbb{R}^2$ with values in $\mathbb{H}$ as follows
\begin{equation*}
 <f,g>_{L^{2}(\mathbb{R}^2,\mathbb{H})}=\int_{\mathbb{R}^2}f(x)\overline{g(x)}dx
\end{equation*}

If $f=g$ we obtain the associated norm by
\begin{equation*}
\|f\|^{2}_{2}=<f,f>_{2}=\int_{\mathbb{R}^2}|f(x)|^2dx
\end{equation*}
The space $L^{2}(\mathbb{R}^2,\mathbb{H})$ is then defined as
\begin{equation*}
L^2(\mathbb{R}^2,\mathbb{H})=\{f| f:\mathbb{R}^2\rightarrow\mathbb{H}, \|f\|_{2}<\infty\}
\end{equation*}

And we define the norm of $L^{2}(\mathbb{R}^2,\mathbb{H})$ by
\begin{equation*}
 \|f\|^{2}_{L^{2}(\mathbb{R}^2,\mathbb{H})}=\|f\|^{2}_{2}
\end{equation*}

\section{The Gabor quaternion linear canonical transform}\label{section3}

The first definition of the QLCT is given by Morais et al. \cite{Morais}. They consider two real  matrixes
\[A_1=\left[ \begin{array}{cc}
a_1 & b_1 \\
c_1 & d_1 \end{array}
\right], A_1=\left[ \begin{array}{cc}
a_2 & b_2 \\
c_2 & d_2 \end{array}
\right]\in  {\mathbb R}^{2\times 2}.\]

 with \ \ \ $a_1d_1-b_1c_1=1,\ a_2d_2-b_2c_2=1,$\\
In \cite{Hitzer2} E. Hitzer  gives a generalization of  the definitions of \cite{Morais} into the case of
 two-sided QLCT of signals $f \in L^1({\mathbb R}^2,{\mathbb H})$, this generalization is defined by

\begin{equation}\label{QLCT}
{\mathcal L}^{i,j }_{A_1,A_2}\{f\}(u)=
\left\{
  \begin{array}{ll}
    \int_{{\mathbb R}^2}{K^{i }_{A_1}\left(t_1,u_1\right)}f\left(t\right)K^{j }_{A_2}\left(t_2,u_2\right)dt,\quad  b_1,b_2\ne 0; \\
    \,   \\
    \sqrt{d_1} e^{i \frac{c_1d_1}{2}{u_1}^2}f(d_1 u_1,t_2)K^{j }_{A_2}\left(t_2,u_2\right), \quad b_1=0,b_2\ne 0; \\
    \, \\
    \sqrt{d_2}{K^{i }_{A_1}\left(t_1,u_1\right)}f(t_1,d_2u_2)e^{j \frac{c_2d_2}{2}u_2^2},  \quad b_1\ne 0, b_2=0; \\
    \, \\
    \sqrt{d_1d_2}e^{i \frac{c_1d_1}{2}{u_1}^2}f(d_1u_1,d_2u_2)e^{j \frac{c_2d_2}{2}u_2^2}, \quad b_1=b_2=0.\\
  \end{array}
\right.
\end{equation}

\subsection{Quaternion linear canonical transform}\leavevmode\par
We recall some important harmonic analysis properties of the QLCT (see \cite{Kou}).
\begin{theorem}[Inversion formula for QLCT]\label{inversion-QLCT}
  Suppose $f\in L^1(\mathbb{R}^2,\mathbb{H})$, then the inversion of two-sided QLCT of $f$ is given by:
\begin{eqnarray*}
  f(x) &=& \{\mathcal{L}_{A_1,A_2}^{-1}\}\{f\}(x) \\
   &=& \int_{\mathbb{R}^2} K^{-i}_{A_1}(x_1,u_1)\mathcal{L}_{A_1,A_2}^{\mathbb{H}}\{f\}(u)K^{-j}_{A_2}(x_2,u_2)du
\end{eqnarray*}
\end{theorem}

\begin{theorem}[Plancherel for QLCT]
  Every 2D quaternion-valued signal $f\in L^{2}(\mathbb{R}^2,\mathbb{H}) $ and its QLCT are related to the Placherel identity in the following way
\begin{equation}\label{parseval-QLCT}
\|f\|_{L^{2}(\mathbb{R}^2,\mathbb{H})}=\|\mathcal{L}_{A_1,A_2}^{\mathbb{H}}\{f\}\|_{Q,2}
\end{equation}
\end{theorem}

\subsection{Gabor quaternion linear canonical transform}

\begin{definition}
  Let $\varphi\in L^{2}(\mathbb{R}^2,\mathbb{H})$ a fixed quaternion windowed function. For $f\in L^{2}(\mathbb{R}^2,\mathbb{H})$ we define the windowed quaternion linear canonical transform by,
  \begin{equation}\label{defi-win-QLCT}
    \displaystyle\mathcal{G}^{\varphi}_{A_1,A_2}\{f\}(\omega,y)=\int_{\mathbb{R}^2}K^{i}_{A_1}(x_1,\omega_1)f(x)\overline{\varphi(x-y)}K^{j}_{A_2}(x_2,\omega_2)dx
  \end{equation}
  with,
  $$K^{i}_{A_1}(x_1,\omega_1)=\left\{
                              \begin{array}{ll}
            \frac{1}{\sqrt{2\pi b_1}}e^{\frac{i}{2}((\frac{a_1}{b_1})x_1^2-(\frac{2}{b_1})x_1\omega_1+(\frac{d_1}{b_1})\omega_1^2-\frac{\pi}{2})}, \quad for~~ b_1\neq0; \\
                                \sqrt{d_1}e^{i(\frac{c_1d_1}{2})\omega_1^2}, \quad  for~~ b_1=0.
                              \end{array}
                            \right.$$

$$K^{j}_{A_2}(x_2,\omega_2)=\left\{
                              \begin{array}{ll}
            \frac{1}{\sqrt{2\pi b_2}}e^{\frac{j}{2}((\frac{a_2}{b_2})x_2^2-(\frac{2}{b_2})x_2\omega_2+(\frac{d_2}{b_2})\omega_2^2-\frac{\pi}{2})}, \quad for~~ b_2\neq0; \\
                                \sqrt{d_2}e^{j(\frac{c_2d_2}{2})\omega_2^2}, \quad  for~~ b_2=0.
                              \end{array}
                            \right.$$

\end{definition}

Without loss of generality we will deal with the case when $b_1b_2\neq0$, and as a special case, when
$A_ 1= A_2 = (a_i,b_i,c_i,d_i) =(0, 1, -1, 0)$, for $i=1,2$, the GQLCT
definition \ref{defi-win-QLCT} will lead to the GQFT definition (see \cite{Elkassimi}).

\begin{remark}\label{rem-liai}
  We can see that, the relationship between the Gabor quaternion linear canonical transform and the linear canonical transform is given by,
\begin{equation}\label{liaison-GQLCT-LCT}
  \mathcal{G}^{\varphi}_{A_1,A_2}\{f\}(\omega,y)=\mathcal{L}_{A_1,A_2}^{\mathbb{H}}\{f(.)\overline{\varphi(.-y)}\}(\omega)
\end{equation}

\end{remark}
Now, we derive many properties harmonic analysis for the GQFT,
\subsection{Inversion formula}

\begin{theorem}
  For $\varphi\in L^{2}(\mathbb{R}^2,\mathbb{H})$ and $f\in L^{2}(\mathbb{R}^2,\mathbb{H})$, we have

\begin{equation}\label{invers-GQLCT}
f(x)=\frac{1}{\|\varphi\|_{L^{2}(\mathbb{R}^2,\mathbb{H})}}\int_{\mathbb{R}^2}\int_{\mathbb{R}^2}K^{-i}_{A_1}(x_1,\omega_1) \mathcal{G}^{\varphi}_{A_1,A_2}\{f\}(\omega,y)K^{-j}_{A_2}(x_2,\omega_2)\varphi(x-y)d\omega dy
\end{equation}
\end{theorem}
\begin{proof}
  We use the relation between $\mathcal{G}^{\varphi}_{A_1,A_2}$ and $\mathcal{L}_{A_1,A_2}^{\mathbb{H}}$ \eqref{liaison-GQLCT-LCT}cited in the remark \ref{rem-liai} and the inversion formula for $\mathcal{L}_{A_1,A_2}^{\mathbb{H}}$ we get
\begin{equation}\label{equ1}
f(x)\overline{\varphi(x-y)}=\int_{\mathbb{R}^2}K^{-i}_{A_1}(x_1,\omega_1) \mathcal{G}^{\varphi}_{A_1,A_2}\{f\}(\omega,y)K^{-j}_{A_2}(x_2,\omega_2)d\omega \end{equation}
we multiple the both sides of \eqref{equ1} by $\varphi(x-y)$ we get
\begin{equation}\label{equ2}
f(x)|\varphi(x-y)|^2=\int_{\mathbb{R}^2}K^{-i}_{A_1}(x_1,\omega_1) \mathcal{G}^{\varphi}_{A_1,A_2}\{f\}(\omega,y)K^{-j}_{A_2}(x_2,\omega_2)d\omega\varphi(x-y)
\end{equation}
integrating both sides of \eqref{equ2} by respecting $dy$ and using Fubini's theorem we obtain
\begin{equation}\label{equ3}
f(x)\int_{\mathbb{R}^2}|\varphi(x-y)|^2 dy=\int_{\mathbb{R}^2}\int_{\mathbb{R}^2}K^{-i}_{A_1}(x_1,\omega_1) \mathcal{G}^{\varphi}_{A_1,A_2}\{f\}(\omega,y)K^{-j}_{A_2}(x_2,\omega_2)\varphi(x-y)d\omega dy
\end{equation}
by substitution in the left hand side, we have
\begin{equation}
f(x)=\frac{1}{\|\varphi\|_{L^{2}(\mathbb{R}^2,\mathbb{H})}}\int_{\mathbb{R}^2}\int_{\mathbb{R}^2}K^{-i}_{A_1}(x_1,\omega_1) \mathcal{G}^{\varphi}_{A_1,A_2}\{f\}(\omega,y)K^{-j}_{A_2}(x_2,\omega_2)\varphi(x-y)d\omega dy
\end{equation}

\end{proof}

\subsection{Plancherel}
\begin{theorem}
For $f, \varphi\in L^{2}(\mathbb{R}^2,\mathbb{H})$, we have the equality
\begin{equation}\label{planch-GQLCT}
  \int_{\mathbb{R}^2}\int_{\mathbb{R}^2}|\mathcal{G}^{\varphi}_{A_1,A_2}\{f\}(\omega,y)|^2_Q d\omega dy =\|f\|^2_{L^{2}(\mathbb{R}^2,\mathbb{H})}\|\varphi\|^2_{L^{2}(\mathbb{R}^2,\mathbb{H})}
\end{equation}
\end{theorem}

\begin{proof}
  We have \begin{eqnarray*}
            \int_{\mathbb{R}^2}\int_{\mathbb{R}^2}|\mathcal{G}^{\varphi}_{A_1,A_2}\{f\}(\omega,y)|^2_Q d\omega dy &=& \int_{\mathbb{R}^2}\int_{\mathbb{R}^2}|\mathcal{L}_{A_1,A_2}^{\mathbb{H}}\{f(.)\overline{\varphi(.-y)}\}(\omega)|^2_Q d\omega dy \\
             &=& \int_{\mathbb{R}^2} \|\mathcal{L}_{A_1,A_2}^{\mathbb{H}}\{f(.)\overline{\varphi(.-y)}\}\|_{Q,2}dy\\
             &=& \int_{\mathbb{R}^2} \|f(.)\overline{\varphi(.-y)}\|^2_{L^{2}(\mathbb{R}^2,\mathbb{H})}dy\\
             &=& \int_{\mathbb{R}^2}\int_{\mathbb{R}^2} |f(x)\overline{\varphi(x-y)}|^2 dxdy
          \end{eqnarray*}
 by substitution $x-y=u$ and Fubini's theorem we get
 \begin{eqnarray*}\int_{\mathbb{R}^2}\int_{\mathbb{R}^2}|\mathcal{G}^{\varphi}_{A_1,A_2}\{f\}(\omega,y)|^2_Q d\omega dy&=& \int_{\mathbb{R}^2}|f(x)|^2 dx  \int_{\mathbb{R}^2}|\varphi(y)|^2dy\\
 &=& \|f\|^2_{L^{2}(\mathbb{R}^2,\mathbb{H})}\|\varphi\|^2_{L^{2}(\mathbb{R}^2,\mathbb{H})}
 \end{eqnarray*}

\end{proof}

\section{Uncertainty principles}\label{section4}
The uncertainty principles play an important role in the quantum physics, these principles  stat that, we cannot give the position and momentum of particle simultaneously with high precision, but only in probabilistic sense with a certain uncertainty. The formulation mathematics of this concept is that, the nonzero function and its Fourier transform cannot both be small. Many theorems are devoted to explain this idea, the first one is the Heisenberg inequality \cite{Heisenberg}, and also a lot of others theorems (see \cite{Hardy, Morgan}). In \cite{Benedicks, Donoho, Slepian} the authors have given attention to relation between the support of the function and its Fourier transform. Our aim in this paper is to demonstrate  Heisenberg's, Logarithmic's and Lieb's inequalities versions.  The last subsection is devoted to prove locally uncertainty principles in the case of WQLCT.

\subsection{Heisenberg Inequality}

In the following theorem we prove a Heisenberg type uncertainty inequality for the GQLCT
\begin{theorem}
  Let $s>0$. There exists a constant $C_{s}>0$ such that, for all $f,\varphi\in L^{2}(\mathbb{R}^2,\mathbb{H})$.
\begin{align}\label{heisen-inequa}
  \left(\int\int_{\mathbb{R}^2\times\mathbb{R}^2}{|\omega|}^{2s}|\mathcal{G}^{\varphi}_{A_1,A_2}\{f\}(\omega,y)|^2d\omega dy\right)^{\frac{1}{2}}& \left(\int\int_{\mathbb{R}^2\times\mathbb{R}^2}{|y|}^{2s}|\mathcal{G}^{\varphi}_{A_1,A_2}\{f\}(\omega,y)|^2d\omega dy\right)^{\frac{1}{2}}\geq \nonumber\\
   & C_{s} \|f\|_{L^2(\mathbb{R}^2,\mathbb{H})}\|\varphi\|_{L^2(\mathbb{R}^2,\mathbb{H})}
\end{align}

\end{theorem}
\begin{proof}
  From the fact that $|a+b|^{s}\leq 2^s(|a|^s+|b|^s)$, we deduce by inequality \eqref{concentration2}
\begin{align}\label{heisen1}
   \int\int_{\mathbb{R}^2\times\mathbb{R}^2}{|\omega|}^{2s}|\mathcal{G}^{\varphi}_{A_1,A_2}\{f\}(\omega,y)|^2d\omega dy &+\int\int_{\mathbb{R}^2\times\mathbb{R}^2}{|y|}^{2s}|\mathcal{G}^{\varphi}_{A_1,A_2}\{f\}(\omega,y)|^2d\omega dy\geq\nonumber  \\
   &\frac{C_{s}}{2^s}  \|f\|^2_{L^2(\mathbb{R}^2,\mathbb{H})}\|\varphi\|^2_{L^2(\mathbb{R}^2,\mathbb{H})}
\end{align}
Now, for every positive real number $t$ the dilates $f_{t}$ and $\varphi_{t}$ belong to $L^2(\mathbb{R}^2,\mathbb{H})$ and $\varphi_{t}$ is a nonzero function, then by the relation \eqref{heisen1} we have
\begin{align}\label{heisen2}
   \int\int_{\mathbb{R}^2\times\mathbb{R}^2}{|\omega|}^{2s}|\mathcal{G}^{\varphi}_{A_1,A_2}\{f_{t}\}(\omega,y)|^2d\omega dy & +\int\int_{\mathbb{R}^2\times\mathbb{R}^2}{|y|}^{2s}|\mathcal{G}^{\varphi}_{A_1,A_2}\{f_{t}\}(\omega,y)|^2d\omega dy\geq \nonumber  \\
   &  \frac{C_{s}}{2^s}  \|f_{t}\|^2_{L^2(\mathbb{R}^2,\mathbb{H})}\|\varphi\|^2_{L^2(\mathbb{R}^2,\mathbb{H})}
\end{align}

hence for every positive real number t
\begin{align}\label{heisen3}
t^{2s} \int\int_{\mathbb{R}^2\times\mathbb{R}^2}{|\omega|}^{2s}|\mathcal{G}^{\varphi}_{A_1,A_2}\{f\}(\omega,y)|^2d\omega dy&+t^{-2s}\int\int_{\mathbb{R}^2\times\mathbb{R}^2}{|y|}^{2s}|\mathcal{G}^{\varphi}_{A_1,A_2}\{f\}(\omega,y)|^2d\omega dy \geq\nonumber\\
   & \frac{C_{s}}{2^s}  \|f\|^2_{L^2(\mathbb{R}^2,\mathbb{H})}\|\varphi\|^2_{L^2(\mathbb{R}^2,\mathbb{H})}
\end{align}

in particular case, when we take
$$t=\left(\frac{\int\int_{\mathbb{R}^2\times\mathbb{R}^2}{|y|}^{2s}|\mathcal{G}^{\varphi}_{A_1,A_2}\{f\}(\omega,y)|^2d\omega dy}{\int\int_{\mathbb{R}^2\times\mathbb{R}^2}{|\omega|}^{2s}|\mathcal{G}^{\varphi}_{A_1,A_2}\{f\}(\omega,y)|^2d\omega dy}\right)^{\frac{1}{4s}}$$
which implies that,
\begin{align*}
  \left(\int\int_{\mathbb{R}^2\times\mathbb{R}^2}{|\omega|}^{2s}|\mathcal{G}^{\varphi}_{A_1,A_2}\{f\}(\omega,y)|^2d\omega dy\right)^{\frac{1}{2}}&\left(\int\int_{\mathbb{R}^2\times\mathbb{R}^2}{|y|}^{2s}|\mathcal{G}^{\varphi}_{A_1,A_2}\{f\}(\omega,y)|^2d\omega dy\right)^{\frac{1}{2}}\geq \\
   &C_{s}\|f\|_{L^2(\mathbb{R}^2,\mathbb{H})}\|\varphi\|_{L^2(\mathbb{R}^2,\mathbb{H})}
\end{align*}

\end{proof}

\subsection{Logarithmic inequality}\leavevmode\par

We start by giving some notations and definitions
\begin{definition}
  A couple $\alpha=(\alpha_1,\alpha_2)$ of non negative integers is called a multi-index. One denotes
  $$|\alpha|=\alpha_1+\alpha_2 ~~ and ~~ \alpha!=\alpha_1!\alpha_2!$$
  and, for $x\in \mathbb{R}^2$
  $$x^{\alpha}=x_1^{\alpha_1}x_2^{\alpha_2}$$
  Derivatives are conveniently expressed by multi-indices
 $$ \partial^{\alpha}=\frac{\partial^{|\alpha|}}{\partial x_1^{\alpha_1}\partial x_2^{\alpha_2}}$$

\end{definition}
Next, we give the Schwartz space as (\cite{Kou1})
$$\mathcal{S}(\mathbb{R}^2,\mathbb{H})=\{f\in C^{\infty}(\mathbb{R}^2,\mathbb{H}): sup_{x\in\mathbb{R}^2}(1+|x|^k)|\partial^{\alpha}f(x)|<\infty\},$$
where  $C^{\infty}(\mathbb{R}^2,\mathbb{H})$ is the set of smooth function from $\mathbb{R}^2$ to $\mathbb{H}$.\\

First we reminder the Logarithmic uncertainty principle for QLCT (see \cite{Mawardi}).
\begin{theorem}
  Let $f\in L^{2}(\mathbb{R}^2,\mathbb{H})$ and $\mathcal{L}_{A_1,A_2}^{\mathbb{H}}\{f\}\in L^{2}(\mathbb{R}^2,\mathbb{H})$ be the QLCT of $f$. Then,
  \begin{equation}\label{loga-QLCT}
\int_{\mathbb{R}^2}ln|x||f(x)|^2 dx +\int_{\mathbb{R}^2}ln|\omega||\mathcal{L}_{A_1,A_2}^{\mathbb{H}}\{f\}(\omega)|^2d\omega\geq(D+ln|b|)\int_{\mathbb{R}^2}|f(x)|^2dx
  \end{equation}
  $D=\psi(\frac{1}{2})-ln\pi$ $\psi(t)=\frac{d}{dt}ln(\Gamma(t))$ where $\Gamma$ is the Gamma function.
\end{theorem}

before to give the Logarithmic's theorem for the GQLCT, we need the following results,
\begin{lemma}\label{lemma-logarithmic}
  Let $\varphi \in S(\mathbb{R}^2,\mathbb{H})$ a windowed quaternionic function and $f\in S(\mathbb{R}^2,\mathbb{H}) $. \\
  We have
  \begin{equation}\label{lemma-loga}
   \int_{\mathbb{R}^2} \int_{\mathbb{R}^2}ln|x||{\mathcal{L}_{A_1,A_2}^{\mathbb{H}}}^{-1}\{\mathcal{G}^{\varphi}_{A_1,A_2}\{f\}(\omega,y)\}(x)|^2dxdy=\|\varphi\|^2_{L^{2}(\mathbb{R}^2,\mathbb{H})}\int_{\mathbb{R}^2}ln|x||f(x)|^2dx
  \end{equation}
\end{lemma}
\begin{proof}
  We have
\begin{eqnarray*}
  \int_{\mathbb{R}^2} \int_{\mathbb{R}^2}ln|x||{\mathcal{L}_{A_1,A_2}^{\mathbb{H}}}^{-1}\{\mathcal{G}^{\varphi}_{A_1,A_2}\{f\}(\omega,y)\}(x)|^2dxdy &=&  \int_{\mathbb{R}^2} \int_{\mathbb{R}^2}ln|x||f(x)\overline{\varphi(x-y)}|^2dxdy \\
   &=&   \int_{\mathbb{R}^2} \int_{\mathbb{R}^2}ln|x||f(x)|^2|\overline{\varphi(x-y)}|^2dxdy\\
   &=& \int_{\mathbb{R}^2} ln|x||f(x)|^2\left(\int_{\mathbb{R}^2}|\overline{\varphi(x-y)}|^2dy\right)dx \\
   &=& \|\varphi\|^2_{L^{2}(\mathbb{R}^2,\mathbb{H})}\int_{\mathbb{R}^2} ln|x||f(x)|^2 dx
\end{eqnarray*}

\end{proof}

\begin{corollary}
  For $\varphi \in S(\mathbb{R}^2,\mathbb{H})$ a windowed quaternionic function and $f\in S(\mathbb{R}^2,\mathbb{H})$, we have
  \begin{equation}\label{corolo-logar}
  \int_{\mathbb{R}^2}|{\mathcal{L}_{A_1,A_2}^{\mathbb{H}}}^{-1}\mathcal{L}_{A_1,A_2}^{\mathbb{H}}\{f\}(x)|^2dx+ \int_{\mathbb{R}^2}ln|\omega||\mathcal{L}_{A_1,A_2}^{\mathbb{H}}\{f\}(\omega)|^2d\omega\geq(D+ln|b|)\int_{\mathbb{R}^2}|\mathcal{L}_{A_1,A_2}^{\mathbb{H}}\{f\}(\omega)|^2d\omega
   \end{equation}
\end{corollary}
Now, we are going to give the Logarithmic uncertainty principle for GQLCT.
\begin{theorem}[Logarithmic inequality for GQCLT]\leavevmode\par
  Let $\varphi\in S(\mathbb{R}^2,\mathbb{H})$ a quaternion windowed function and $f\in S(\mathbb{R}^2,\mathbb{H})$, we have the following logarithmic inequality:
  \begin{equation}\label{logar-inequality}
  \|\varphi\|^2_{L^{2}(\mathbb{R}^2,\mathbb{H})}\int_{\mathbb{R}^2} ln|x||f(x)|^2 dx+\int_{\mathbb{R}^2}\int_{\mathbb{R}^2} ln|\omega||\mathcal{G}^{\varphi}_{A_1,A_2}\{f\}(\omega,y)|^2 d\omega dy\geq  \|\varphi\|^2_{L^{2}(\mathbb{R}^2,\mathbb{H})}(D+ln|b|)\int_{\mathbb{R}^2} |f(x)|^2 dx
  \end{equation}
\end{theorem}
\begin{proof}
  For $\varphi,f\in S(\mathbb{R}^2,\mathbb{H})$ we have by corollary \ref{corolo-logar}
   \begin{equation}\label{eq1}
  \int_{\mathbb{R}^2}|{\mathcal{L}_{A_1,A_2}^{\mathbb{H}}}^{-1}\mathcal{L}_{A_1,A_2}^{\mathbb{H}}\{f\}(x)|^2dx+ \int_{\mathbb{R}^2}ln|\omega||\mathcal{L}_{A_1,A_2}^{\mathbb{H}}\{f\}(\omega)|^2d\omega\geq(D+ln|b|)\int_{\mathbb{R}^2}|\mathcal{L}_{A_1,A_2}^{\mathbb{H}}\{f\}(\omega)|^2d\omega
   \end{equation}
   we replace $f$ by $f(.)\overline{\varphi(.-y)}$ in \eqref{eq1} we get
   \begin{equation}\label{eq2}
  \int_{\mathbb{R}^2}|{\mathcal{L}_{A_1,A_2}^{\mathbb{H}}}^{-1}\{\mathcal{G}^{\varphi}_{A_1,A_2}\{f\}(.,y)\}(x)|^2dx+ \int_{\mathbb{R}^2}ln|\omega||\mathcal{G}^{\varphi}_{A_1,A_2}\{f\}(\omega,y)|^2d\omega\geq(D+ln|b|)\int_{\mathbb{R}^2}|\mathcal{G}^{\varphi}_{A_1,A_2}\{f\}(\omega,y)|^2d\omega
   \end{equation}
Integrating both sides of inequality \eqref{eq2} with respect to $dy$, we obtain
\begin{multline}\label{eq3}
 \int_{\mathbb{R}^2} \int_{\mathbb{R}^2}|{\mathcal{L}_{A_1,A_2}^{\mathbb{H}}}^{-1}\{\mathcal{G}^{\varphi}_{A_1,A_2}\{f\}(.,y)\}(x)|^2dxdy+ \int_{\mathbb{R}^2}\int_{\mathbb{R}^2}ln|\omega||\mathcal{G}^{\varphi}_{A_1,A_2}\{f\}(\omega,y)|^2d\omega dy \geq  \\
 (D+ln|b|)\int_{\mathbb{R}^2}\int_{\mathbb{R}^2}|\mathcal{G}^{\varphi}_{A_1,A_2}\{f\}(\omega,y)|^2d\omega dy
\end{multline}
 Applying the lemma \ref{lemma-logarithmic} into the first term on the left hand side of \eqref{eq3}, and the Plancherel formula for the second terms of the left hand side and for right hand side, we obtain our result
   \begin{equation}
  \|\varphi\|^2_{L^{2}(\mathbb{R}^2,\mathbb{H})}\int_{\mathbb{R}^2} ln|x||f(x)|^2 dx+\int_{\mathbb{R}^2}\int_{\mathbb{R}^2} ln|\omega||\mathcal{G}^{\varphi}_{A_1,A_2}\{f\}(\omega,y)|^2 d\omega dy\geq  \|\varphi\|^2_{L^{2}(\mathbb{R}^2,\mathbb{H})}(D+ln|b|)\int_{\mathbb{R}^2} |f(x)|^2 dx
  \end{equation}

\end{proof}
\subsection{Lieb's inequality}\leavevmode\par
In \cite{Lieb}, the authors proved a new inequality in the case of Wigner-Ville distribution, in this paper, we will prove a version of the Lieb's theorem for the GQFT.
\begin{lemma}
  Let $\varphi\in L^p(\mathbb{R}^2,\mathbb{H})$, $f\in L^q(\mathbb{R}^2,\mathbb{H})$ and $p,q \in[1,+\infty[$ with $\frac{1}{p}+\frac{1}{q}=1$, we have
\begin{equation}\label{young-inequality}
\|\mathcal{G}^{\varphi}_{A_1,A_2}\{f\}(\omega,y)\|_{\infty} \leq \frac{|b_1b_2|^{-\frac{1}{2}}}{2\pi}\|f\|_{L^q(\mathbb{R}^2,\mathbb{H})}\|\varphi\|_{L^p(\mathbb{R}^2,\mathbb{H})}
\end{equation}
\end{lemma}
 \begin{proof}
   We have
\begin{eqnarray*}
  |\mathcal{G}^{\varphi}_{A_1,A_2}\{f\}(\omega,y)| &=& |\int_{\mathbb{R}^2}K^{i}_{A_1}(x_1,\omega_1)f(x)\overline{\varphi(x-y)}K^{j}_{A_2}(x_2,\omega_2)dx| \\
      &\leq& \frac{|b_1b_2|^{-\frac{1}{2}}}{2\pi}\int_{\mathbb{R}^2} |f(x)||\overline{\varphi(x-y)}|dx
\end{eqnarray*}
Using H\"{o}lder inequality we get our result
\begin{equation*}
\|\mathcal{G}^{\varphi}_{A_1,A_2}\{f\}(\omega,y)\|_{\infty} \leq \frac{|b_1b_2|^{-\frac{1}{2}}}{2\pi}\|f\|_{L^q(\mathbb{R}^2,\mathbb{H})}\|\varphi\|_{L^p(\mathbb{R}^2,\mathbb{H})}
\end{equation*}
\end{proof}
In \cite{Bahri3} we have the following Hausdorff-Young inequality for the QLCT,
\begin{theorem}[Hausdorff-Young inequality]\leavevmode\par
Let $1\leq p\leq 2 $ and letting $p^{'}$ be such that $\frac{1}{p}+\frac{1}{p^{'}}=1$, then, for all $f\in L^{p}$ we have
\begin{equation}\label{hausdorff-QLCT}
  \|\mathcal{L}_{A_1,A_2}^{\mathbb{H}}\{f\}\|_{q,p^{'}}\leq \frac{|b_1b_2|^{\frac{-1}{2}+\frac{1}{p^{'}}}}{2\pi}\|f\|_{L^p(\mathbb{R}^2,\mathbb{H})}
\end{equation}
  where,
  $$\|\mathcal{L}_{A_1,A_2}^{\mathbb{H}}\{f\}\|_{q,p^{'}}=\left( \int_{\mathbb{R}^2}|\mathcal{L}_{A_1,A_2}^{\mathbb{H}}\{f\}(\omega)|_{q}^{p^{'}}d\omega\right)^{\frac{1}{p^{'}}}$$
with
$$|\mathcal{L}_{A_1,A_2}^{\mathbb{H}}\{f\}(\omega)|_{q}=|\mathcal{L}_{A_1,A_2}^{\mathbb{H}}\{f_0\}(\omega)|+|\mathcal{L}_{A_1,A_2}^{\mathbb{H}}\{f_1\}(\omega)|+|\mathcal{L}_{A_1,A_2}^{\mathbb{H}}\{f_2\}(\omega)|+|\mathcal{L}_{A_1,A_2}^{\mathbb{H}}\{f_3\}(\omega)|$$
\end{theorem}

Now, we give a version of Lieb's inequality for the GQLCT,

\begin{theorem}[Lieb's inequality for GQLCT]\label{lieb-theorem}\leavevmode\par
  Let $2\leq p<\infty$ and $f,\varphi\in L^{2}(\mathbb{R}^2,\mathbb{H})$. For $p^{'}>0$ with $\frac{1}{p}+\frac{1}{p^{'}}=1$ there is a positive constant $C$ such that
\begin{equation}\label{lieb-inequality}
  \int_{\mathbb{R}^2}\int_{\mathbb{R}^2}|\mathcal{G}^{\varphi}_{A_1,A_2}\{f\}(\omega,y)|^{p^{'}}d\omega dy\leq C\frac{|b_1b_2|^{\frac{-p^{'}}{2}+1}}{(2\pi)^{p^{'}}}\|f\|^{p^{'}}_{L^2(\mathbb{R}^2,\mathbb{H})}\|\varphi\|^{p^{'}}_{L^2(\mathbb{R}^2,\mathbb{H})}
\end{equation}
\end{theorem}
\begin{proof}
  We have $\mathcal{G}^{\varphi}_{A_1,A_2}\{f\}(\omega,y)=L^{\mathbb{H}}_{A_1,A_2}\{f(.)\overline{\varphi(.-y)}\}(\omega)$
  by the theorem of Hausdorff-Young associated with $\mathcal{L}_{A_1,A_2}^{\mathbb{H}}$ in theorem \ref{hausdorff-QLCT}, we have
  \begin{eqnarray*}
    \left(\int_{\mathbb{R}^2}|\mathcal{G}^{\varphi}_{A_1,A_2}\{f\}(\omega,y)|^{p^{'}}d\omega\right)^{\frac{1}{p^{'}}} &=& \left(\int_{\mathbb{R}^2}|L^{\mathbb{H}}_{A_1,A_2}\{f(.)\overline{\varphi(.-y)}\}(\omega)|^{p^{'}}d\omega\right)^{\frac{1}{p^{'}}} \\
     &\leq& \frac{|b_1b_2|^{\frac{-1}{2}+\frac{1}{p^{'}}}}{2\pi}\|f(.)\overline{\varphi(.-y)}\|_{L^p(\mathbb{R}^2,\mathbb{H})} \\
     &=& \frac{|b_1b_2|^{\frac{-1}{2}+\frac{1}{p^{'}}}}{2\pi}\left(\int_{\mathbb{R}^2}|f(x)\overline{\varphi(x-y)}|^pdx\right)^{\frac{1}{p}}
  \end{eqnarray*}
  then,
 \begin{equation}\label{eqq1}
   \int_{\mathbb{R}^2}|\mathcal{G}^{\varphi}_{A_1,A_2}\{f\}(\omega,y)|^{p^{'}}d\omega \leq \frac{|b_1b_2|^{\frac{-p^{'}}{2}+1}}{(2\pi)^{p^{'}}} \left(\int_{\mathbb{R}^2}|f(x)\overline{\varphi(x-y)}|^pdx\right)^{\frac{p^{'}}{p}}
 \end{equation}

Integrating both sides of \eqref{eqq1} with respect to $dy$, we get
\begin{equation}\label{eqq2}
\int_{\mathbb{R}^2}\int_{\mathbb{R}^2}|\mathcal{G}^{\varphi}_{A_1,A_2}\{f\}(\omega,y)|^{p^{'}}d\omega dy\leq \frac{|b_1b_2|^{\frac{-p^{'}}{2}+1}}{(2\pi)^{p^{'}}} \int_{\mathbb{R}^2}\left(|f|^p\ast|\varphi|^p(y)|^p \right)^{\frac{p^{'}}{p}}dy
\end{equation}
then,
\begin{eqnarray*}
\int_{\mathbb{R}^2}\int_{\mathbb{R}^2}|\mathcal{G}^{\varphi}_{A_1,A_2}\{f\}(\omega,y)|^{p^{'}}d\omega dy&\leq&  \frac{|b_1b_2|^{\frac{-p^{'}}{2}+1}}{(2\pi)^{p^{'}}} \int_{\mathbb{R}^2}\left(\int_{\mathbb{R}^2}|f(x)\overline{\varphi(x-y)}|^pdx\right)^{\frac{p^{'}}{p}}dy\\
&=& \frac{|b_1b_2|^{\frac{-p^{'}}{2}+1}}{(2\pi)^{p^{'}}} \||f|^p\ast|\varphi|^p\|^{\frac{p^{'}}{p}}_{\frac{p^{'}}{p}}
\end{eqnarray*}
 Applying the Young's inequality \cite{Groki,Lieb} to the functions $|f|^p; |\varphi|^p$ which are element of $L^{\frac{2}{p}}(\mathbb{R}^2,\mathbb{H})$, with
 $(s,s,t)=(\frac{2}{p},\frac{2}{p},\frac{p^{'}}{p}),\quad(\frac{1}{s}+\frac{1}{s}=1+\frac{1}{t})$, we obtain
$$\||f|^p\ast|\varphi|^p\|_t\leq C_sC_sC_{t^{'}}\||f|^p\|_{s^{'}}\||\varphi|^p\|_s=C_sC_sC_{t^{'}}\|f\|^p_2\|\varphi\|^p_2$$
where $C_s=(s^{\frac{1}{s}}(s^{'})^{1/s^{'}})\quad C_{t^{'}}=((t^{'})^{1/t^{'}}t^{-1/t})$ and $\frac{1}{s}+\frac{1}{s^{'}}=1, \quad \frac{1}{t}+\frac{1}{t^{'}}=1$
Therefore
$$
\int_{\mathbb{R}^2}\int_{\mathbb{R}^2}|\mathcal{G}^{\varphi}_{A_1,A_2}\{f\}(\omega,y)|^{p^{'}}d\omega dy
\leq \frac{|b_1b_2|^{\frac{-p^{'}}{2}+1}}{(2\pi)^{p^{'}}}C_sC_sC_{t^{'}}\|f\|^{p^{'}}_{L^2(\mathbb{R}^2,\mathbb{H})}\|\varphi\|^{p^{'}}_{L^2(\mathbb{R}^2,\mathbb{H})}$$
we have the desired result
$$\int_{\mathbb{R}^2}\int_{\mathbb{R}^2}|\mathcal{G}^{\varphi}_{A_1,A_2}\{f\}(\omega,y)|^{p^{'}}d\omega dy \leq \frac{|b_1b_2|^{\frac{-p^{'}}{2}+1}}{(2\pi)^{p^{'}}}\left(\frac{2}{p^{'}}\right)^{1/p^{'}}\|f\|^{p^{'}}_{L^2(\mathbb{R}^2,\mathbb{H})}\|\varphi\|^{p^{'}}_{L^2(\mathbb{R}^2,\mathbb{H})}$$
\end{proof}
\subsection{Local uncertainty principle}\label{section5}
 \begin{theorem}
   Let $\varphi \in {L^2(\mathbb{R}^2,\mathbb{H})}$ a quaternionic windowed function with $\|\varphi\|_{{L^2(\mathbb{R}^2,\mathbb{H})}}=1$, suppose that $\|f\|_{{L^2(\mathbb{R}^2,\mathbb{H})}}=1$, then for $\Sigma\subset \mathbb{R}^2\times\mathbb{R}^2$ and $\varepsilon\geq0$ such that
   $$\int\int_{\Sigma}|\mathcal{G}^{\varphi}_{A_1,A_2}\{f\}(\omega,y)|^2d\omega dy\geq1-\varepsilon$$
   We have $|b_1b_2|^{\frac{1}{2}}2\pi(1-\varepsilon)\leq m(\Sigma)$.\\
   with $m(\Sigma)$ is the Lebesgue measure of $\Sigma$.
 \end{theorem}
\begin{proof}
  For $f\in {L^2(\mathbb{R}^2,\mathbb{H})}$
  we have by the lemma \ref{young-inequality}

  \begin{equation*}
\|\mathcal{G}^{\varphi}_{A_1,A_2}\{f\}(\omega,y)\|_{L^{\infty}(\mathbb{R}^2,\mathbb{H})} \leq \frac{|b_1b_2|^{-\frac{1}{2}}}{2\pi}\|f\|_{L^q(\mathbb{R}^2,\mathbb{H})}\|\varphi\|_{L^p(\mathbb{R}^2,\mathbb{H})}
\end{equation*}
 From the relation above, we get
 \begin{eqnarray*}
  1-\varepsilon\leq\int\int_{\Sigma}|\mathcal{G}^{\varphi}_{A_1,A_2}\{f\}(\omega,y)|^2d\omega dy &\leq&\|\mathcal{G}^{\varphi}_{A_1,A_2}\{f\}\|_{L^{\infty}(\mathbb{R}^2,\mathbb{H})} m(\Sigma)\\
   &\leq& \frac{|b_1b_2|^{-\frac{1}{2}}}{2\pi}m(\Sigma)\|f\|_{L^q(\mathbb{R}^2,\mathbb{H})}\|\varphi\|_{L^p(\mathbb{R}^2,\mathbb{H})}\\
   &=&\frac{|b_1b_2|^{-\frac{1}{2}}}{2\pi}m(\Sigma)
  \end{eqnarray*}
  then, $$|b_1b_2|^{\frac{1}{2}}2\pi(1-\varepsilon)\leq m(\Sigma)$$

\end{proof}

\begin{theorem}\label{concentration1-theo}
  Let $\Sigma\subset \mathbb{R}^2\times\widehat{\mathbb{R}^2}$ such that $0<m(\Sigma)<1$. Then for all $f, \varphi\in L^2(\mathbb{R}^2,\mathbb{H})$
  \begin{equation}\label{concentration1}
\|f\|_{L^2(\mathbb{R}^2,\mathbb{H})}\|\varphi\|_{L^2(\mathbb{R}^2,\mathbb{H})} \leq   \frac{1}{\sqrt{1-m(\Sigma)}}\| \mathcal{G}^{\varphi}_{A_1,A_2}\{f\}\|_{L^2(\Sigma^{c},\mathbb{H})}
  \end{equation}

\end{theorem}

\begin{proof}
  For every function $f\in L^2(\mathbb{R}^2,\mathbb{H})$, and by the Plancherel equality \eqref{planch-GQLCT} we have,
  \begin{eqnarray*}
    \|\mathcal{G}^{\varphi}_{A_1,A_2}\{f\}\|^2_{L^2(\mathbb{R}^2\times\mathbb{R}^2,\mathbb{H})} &=& \int_{\mathbb{R}^2}\int_{\mathbb{R}^2}|\mathcal{G}^{\varphi}_{A_1,A_2}\{f\}(\omega,y)|^2 d\omega dy \\
     &=& \int\int_{\Sigma}|\mathcal{G}^{\varphi}_{A_1,A_2}\{f\}(\omega,y)|^2 d\omega dy+\in\int_{\Sigma^{c}}|\mathcal{G}^{\varphi}_{A_1,A_2}\{f\}(\omega,y)|^2 d\omega dy \\
     &\leq& m(\Sigma)\|\varphi\|^2_{L^2(\mathbb{R}^2,\mathbb{H})}\|f\|^2_{L^2(\mathbb{R}^2,\mathbb{H})}+ \|\mathcal{G}^{\varphi}_{A_1,A_2}\{f\}\|_{\Sigma^{c},\mathbb{H}}
  \end{eqnarray*}
 Then,  $$\|\mathcal{G}^{\varphi}_{A_1,A_2}\{f\}\|^2_{\Sigma^{c},\mathbb{H})}\geq \|\mathcal{G}^{\varphi}_{A_1,A_2}\{f\}\|^2_{L^2(\mathbb{R}^2\times\mathbb{R}^2,\mathbb{H})}-m(\Sigma)\|\varphi\|^2_{L^2(\mathbb{R}^2,\mathbb{H})}\|f\|^2_{L^2(\mathbb{R}^2,\mathbb{H})}$$

Using the Plancherel formula \eqref{planch-GQLCT}, we get

$$\|f\|_{L^2(\mathbb{R}^2,\mathbb{H})}\|\varphi\|_{L^2(\mathbb{R}^2,\mathbb{H})} \leq   \frac{1}{\sqrt{1-m(\Sigma)}}\| \mathcal{G}^{\varphi}_{A_1,A_2}\{f\}\|_{L^2(\Sigma^{c},\mathbb{H})}$$

\end{proof}
\begin{remark}
  This shows that for a non zero function $f$, if the  Gabor QLCT  $\mathcal{G}^{\varphi}_{A_1,A_2}\{f\}$ is  concentrated on a set $\Sigma$ of volume such that, $0<m(\Sigma)<1 $ then $f\equiv0$ or $\varphi\equiv0$.
\end{remark}
\begin{theorem}
  Let $s>0$. Then there exists a constant $C_{s}>0$ such that, for $f,\varphi\in L^{2}(\mathbb{R}^2,\mathbb{H})$
  \begin{equation}\label{concentration2}
   \|f\|_{L^{2}(\mathbb{R}^2,\mathbb{H})}\|\varphi\|_{L^{2}(\mathbb{R}^2,\mathbb{H})} \leq C_{s}\left(\int\int_{\mathbb{R}^2\times\mathbb{R}^2}
    {|(\omega,y)|}^{2s}|\mathcal{G}^{\varphi}_{A_1,A_2}\{f\}(\omega,y)|^2d\omega dy\right)^{\frac{1}{2}}
  \end{equation}
\end{theorem}

\begin{proof}
  Let $0<r\leq1$ be a real number and $B_{r}=\{(\omega,y)\in \mathbb{R}^2\times\mathbb{R}^2: |(\omega,y)|<r\}$ the ball of center 0 and radius $r$ in $\mathbb{R}^2\times\mathbb{R}^2 $. Fix $0<r_{0}\leq1$ small enough such that $m(B_{r_{0}})<1$. Therefore by inequality \eqref{concentration1} we obtain
  \begin{eqnarray*}
    \|f\|_{L^2(\mathbb{R}^2,\mathbb{H})}\|\varphi\|_{L^2(\mathbb{R}^2,\mathbb{H})}
     &\leq&  \frac{1}{r_{0}^{s}\sqrt{1-m(B_{r_{0}})}}\left(\int\int_{|(\omega,y)|>r_{0}} r_{0}^{2s}|\mathcal{G}^{\varphi}_{A_1,A_2}\{f\}(\omega,y)|^2d\omega dy \right)^{\frac{1}{2}}  \\
     &\leq&  \frac{1}{r_{0}^{s}\sqrt{1-m(B_{r_{0}})}}\left(\int\int_{|(\omega,y)|>r_{0}} {|(\omega,y)|}^{2s}|\mathcal{G}^{\varphi}_{A_1,A_2}\{f\}(\omega,y)|^2d\omega dy\right)^{\frac{1}{2}} \\
     &\leq&\frac{1}{r_{0}^{s}\sqrt{1-m(B_{r_{0}})}}\left(\int\int_{\mathbb{R}^2\times\mathbb{R}^2}{|(\omega,y)|}^{2s}|\mathcal{G}^{\varphi}_{A_1,A_2}\{f\}(\omega,y)|^2
    d\omega dy\right)^{\frac{1}{2}}
  \end{eqnarray*}
    We obtain the desired result by taking $C_{s}=r_{0}^{s}\sqrt{1-m(B_{r_{0}})}$
\end{proof}

\begin{center}

\end{center}

\end{document}